\documentclass[12pt]{amsart}       
\usepackage{txfonts}
\usepackage{amssymb}
\usepackage{eucal}
\usepackage{graphicx}
\usepackage{amsmath}
\usepackage{amscd}
\usepackage[all]{xy}           
\usepackage{amsfonts,latexsym}
\usepackage{xspace}
\usepackage{epsfig}
\usepackage{float}
\usepackage{color}
\usepackage{fancybox}
\usepackage{colordvi}
\usepackage{multicol}
\usepackage{colordvi}
\usepackage{ifpdf}
\ifpdf
  \usepackage[colorlinks,final,backref=page,hyperindex]{hyperref}
\else
  \usepackage[colorlinks,final,backref=page,hyperindex,hypertex]{hyperref}
\fi
\usepackage[active]{srcltx} 

\usepackage{tikz}

\usepackage{graphicx}

\newtheorem{theorem}{Theorem}[section]
\newtheorem{prop}[theorem]{Proposition}
\newtheorem{lemma}[theorem]{Lemma}

\newtheorem{prop-def}{Proposition-Definition}[section]
\newtheorem{coro-def}{Corollary-Definition}[section]

\theoremstyle{definition}
\newtheorem{defn}[theorem]{Definition}
\newtheorem{remark}[theorem]{Remark}

\newtheorem{exam}[theorem]{Example}

\newcommand{\nc}{\newcommand}
\nc{\tred}[1]{\textcolor{red}{#1}}
\nc{\tblue}[1]{\textcolor{blue}{#1}}
\nc{\tgreen}[1]{\textcolor{green}{#1}}
\nc{\tpurple}[1]{\textcolor{purple}{#1}}
\nc{\btred}[1]{\textcolor{red}{\bf #1}}
\nc{\btblue}[1]{\textcolor{blue}{\bf #1}}
\nc{\btgreen}[1]{\textcolor{green}{\bf #1}}
\nc{\btpurple}[1]{\textcolor{purple}{\bf #1}}
\nc{\NN}{{\mathbb N}}
\nc{\ncsha}{{\mbox{\cyr X}^{\mathrm NC}}} \nc{\ncshao}{{\mbox{\cyr
X}^{\mathrm NC}_0}}

\newcommand{\efootnote}[1]{}

\renewcommand{\textbf}[1]{}

\newcommand{\delete}[1]{}

\nc{\mlabel}[1]{\label{#1}}
\nc{\mcite}[1]{\cite{#1}}
\nc{\mref}[1]{\ref{#1}}
\nc{\meqref}[1]{\eqref{#1}}
\nc{\mbibitem}[1]{\bibitem{#1}}

\delete{
\nc{\mlabel}[1]{\label{#1}{\hfill \hspace{1cm}{\bf{{\ }\hfill(#1)}}}}
\nc{\mcite}[1]{\cite{#1}{{\bf{{\ }(#1)}}}}
\nc{\mref}[1]{\ref{#1}{{\bf{{\ }(#1)}}}}
\nc{\meqref}[1]{\eqref{#1}{{\bf{{\ }(#1)}}}}
\nc{\mbibitem}[1]{\bibitem[\bf #1]{#1}}
}
\nc{\opa}{\ast} \nc{\opb}{\odot} \nc{\op}{\bullet} \nc{\pa}{\frakL}
\nc{\arr}{\rightarrow} \nc{\lu}[1]{(#1)} \nc{\mult}{\mrm{mult}}
\nc{\diff}{\mathfrak{Diff}}
\nc{\opc}{\sharp}\nc{\opd}{\natural}
\nc{\ope}{\circ}
\nc{\dpt}{\mathrm{d}}
\nc{\hck}{H_{RT}}
\nc{\vdf}{\calf}
\nc{\ldf}{\calf_\ell}
\nc{\hlf}{H_\ell}
\nc{\onek}{\mathbf{1}_\bfk}
\nc{\tforall}{\, \text{ for }\, }
\nc{\qforall}{\quad \text{for all }}

\nc{\mrba}{MRBA\xspace}
\nc{\Mrba}{MRBA\xspace}
\nc{\mrbas}{MRBAs\xspace}
\nc{\Mrbas}{MRBAs\xspace}
\nc{\match}{matching\xspace}
\nc{\Match}{Matching\xspace}
\nc{\Mza}{Matching Zinbiel algebra\xspace}
\nc{\Mzas}{Matching Zinbiel algebras\xspace}
\nc{\mza}{matching Zinbiel algebra\xspace}
\nc{\mzas}{matching Zinbiel algebras\xspace}

\nc{\za}{Zinbiel algebra\xspace}
\nc{\paybe}{polarized associative Yang-Baxter equation\xspace}
\nc{\Paybe}{Polarized associative Yang-Baxter equation\xspace}
\nc{\cpaybe}{PAYBE}

\nc{\diam}{alternating\xspace}
\nc{\Diam}{Alternating\xspace}
\nc{\cdiam}{canonical alternating\xspace}
\nc{\Cdiam}{Canonical alternating\xspace}
\nc{\AW}{\mathcal{A}}
\nc{\rba}{Rota-Baxter algebra\xspace}
\nc{\rbas}{Rota-Baxter algebras\xspace}

\nc{\ari}{\mathrm{ar}}
\nc{\lef}{\mathrm{lef}}
\nc{\Sh}{\mathrm{ST}}
\nc{\Cr}{\mathrm{Cr}}
\nc{\st}{{Schr\"oder tree}\xspace}
\nc{\sts}{{Schr\"oder trees}\xspace}
\nc{\vertset}{\Omega} 
\nc{\assop}{\quad \begin{picture}(5,5)(0,0)
\line(-1,1){10}
\put(-2.2,-2.2){$\bullet$}
\line(0,-1){10}\line(1,1){10}
\end{picture} \quad \smallskip}

\nc{\operator}{\begin{picture}(5,5)(0,0)
\line(0,-1){6}
\put(-2.6,-1.8){$\bullet$}
\line(0,1){9}
\end{picture}}

\nc{\idx}{\begin{picture}(6,6)(-3,-3)
\put(0,0){\line(0,1){6}}
\put(0,0){\line(0,-1){6}}
\end{picture}}

\nc{\pb}{{\mathrm{pb}}}
\nc{\Lf}{{\mathrm{Lf}}}

\nc{\lft}{{left tree}\xspace}
\nc{\lfts}{{left trees}\xspace}

\nc{\fat}{{fundamental averaging tree}\xspace}

\nc{\fats}{{fundamental averaging trees}\xspace}
\nc{\avt}{\mathrm{Avt}}

\nc{\rass}{{\mathit{RAss}}}

\nc{\aass}{{\mathit{AAss}}}

\nc{\twovec}[2]{\left[#1\atop #2\right]}

\nc{\vin}{{\mathrm Vin}}    
\nc{\lin}{{\mathrm Lin}}    
\nc{\inv}{\mathrm{I}n}
\nc{\gensp}{V} 
\nc{\genbas}{\mathcal{V}} 
\nc{\bvp}{V_P}     
\nc{\gop}{{\,\omega\,}}     

\nc{\bin}[2]{ (_{\stackrel{\scs{#1}}{\scs{#2}}})}  
\nc{\binc}[2]{ \left (\!\! \begin{array}{c} \scs{#1}\\
    \scs{#2} \end{array}\!\! \right )}  
\nc{\bincc}[2]{  \left ( {\scs{#1} \atop
    \vspace{-1cm}\scs{#2}} \right )}  
\nc{\bs}{\bar{S}} \nc{\cosum}{\sqsubset} \nc{\la}{\longrightarrow}
\nc{\rar}{\rightarrow} \nc{\dar}{\downarrow} \nc{\dprod}{**}
\nc{\dap}[1]{\downarrow \rlap{$\scriptstyle{#1}$}}
\nc{\md}{\mathrm{dth}} \nc{\uap}[1]{\uparrow
\rlap{$\scriptstyle{#1}$}} \nc{\defeq}{\stackrel{\rm def}{=}}
\nc{\disp}[1]{\displaystyle{#1}} \nc{\dotcup}{\
\displaystyle{\bigcup^\bullet}\ } \nc{\gzeta}{\bar{\zeta}}
\nc{\hcm}{\ \hat{,}\ } \nc{\hts}{\hat{\otimes}}
\nc{\barot}{{\otimes}} \nc{\free}[1]{\bar{#1}}
\nc{\uni}[1]{\tilde{#1}} \nc{\hcirc}{\hat{\circ}} \nc{\lleft}{[}
\nc{\lright}{]} \nc{\lc}{\lfloor} \nc{\rc}{\rfloor}
\nc{\curlyl}{\left \{ \begin{array}{c} {} \\ {} \end{array}
    \right .  \!\!\!\!\!\!\!}
\nc{\curlyr}{ \!\!\!\!\!\!\!
    \left . \begin{array}{c} {} \\ {} \end{array}
    \right \} }
\nc{\longmid}{\left | \begin{array}{c} {} \\ {} \end{array}
    \right . \!\!\!\!\!\!\!}
\nc{\onetree}{\bullet} \nc{\ora}[1]{\stackrel{#1}{\rar}}
\nc{\ola}[1]{\stackrel{#1}{\la}}
\nc{\ot}{\otimes} \nc{\mot}{{{\boxtimes\,}}}
\nc{\otm}{\overline{\boxtimes}} \nc{\sprod}{\bullet}
\nc{\scs}[1]{\scriptstyle{#1}} \nc{\mrm}[1]{{\rm #1}}
\nc{\margin}[1]{\marginpar{\rm #1}}   
\nc{\dirlim}{\displaystyle{\lim_{\longrightarrow}}\,}
\nc{\invlim}{\displaystyle{\lim_{\longleftarrow}}\,}
\nc{\mvp}{\vspace{0.3cm}} \nc{\tk}{^{(k)}} \nc{\tp}{^\prime}
\nc{\ttp}{^{\prime\prime}} \nc{\svp}{\vspace{2cm}}
\nc{\vp}{\vspace{8cm}} \nc{\proofbegin}{\noindent{\bf Proof: }}
\nc{\proofend}{$\blacksquare$ \vspace{0.3cm}}
\nc{\modg}[1]{\!<\!\!{#1}\!\!>}
\nc{\intg}[1]{F_C(#1)} \nc{\lmodg}{\!
<\!\!} \nc{\rmodg}{\!\!>\!}
\nc{\cpi}{\widehat{\Pi}}
\nc{\sha}{{\mbox{\cyr X}}}  

\newfont{\scyr}{wncyr10 scaled 550}
\nc{\ssha}{\mbox{\bf \scyr X}}
\nc{\shap}{{\mbox{\cyrs X}}} 
\nc{\shpr}{\diamond}    
\nc{\shp}{\ast} \nc{\shplus}{\shpr^+}
\nc{\shprc}{\shpr_c}    
\nc{\msh}{\ast} \nc{\zprod}{m_0} \nc{\oprod}{m_1}
\nc{\vep}{\epsilon} \nc{\labs}{\mid\!} \nc{\rabs}{\!\mid}
\nc{\sqmon}[1]{\langle #1\rangle}

\nc{\mmbox}[1]{\mbox{\ #1\ }} \nc{\dep}{\mrm{dep}} \nc{\fp}{\mrm{FP}}
\nc{\rchar}{\mrm{char}} \nc{\End}{\mrm{End}} \nc{\Fil}{\mrm{Fil}}
\nc{\Mor}{Mor\xspace} \nc{\gmzvs}{gMZV\xspace}
\nc{\gmzv}{gMZV\xspace} \nc{\mzv}{MZV\xspace}
\nc{\mzvs}{MZVs\xspace} \nc{\Hom}{\mrm{Hom}} \nc{\id}{\mrm{id}}
\nc{\im}{\mrm{im}} \nc{\incl}{\mrm{incl}} \nc{\map}{\mrm{Map}}
\nc{\mchar}{\rm char} \nc{\nz}{\rm NZ} \nc{\supp}{\mathrm Supp}

\nc{\Alg}{\mathbf{Alg}} \nc{\Bax}{\mathbf{Bax}} \nc{\bff}{\mathbf f}
\nc{\bfk}{{\bf k}} \nc{\bfone}{{\bf 1}} \nc{\bfx}{\mathbf x}
\nc{\bfy}{\mathbf y}
\nc{\base}[1]{\bfone^{\otimes ({#1}+1)}} 
\nc{\Cat}{\mathbf{Cat}}

\nc{\detail}{\marginpar{\bf More detail}
    \noindent{\bf Need more detail!}
    \svp}
\nc{\Int}{\mathbf{Int}} \nc{\Mon}{\mathbf{Mon}}
\nc{\rbtm}{{shuffle }} \nc{\rbto}{{Rota-Baxter }}
\nc{\remarks}{\noindent{\bf Remarks: }} \nc{\Rings}{\mathbf{Rings}}
\nc{\Sets}{\mathbf{Sets}} \nc{\wtot}{\widetilde{\odot}}
\nc{\wast}{\widetilde{\ast}} \nc{\bodot}{\bar{\odot}}
\nc{\bast}{\bar{\ast}} \nc{\hodot}[1]{\odot^{#1}}
\nc{\hast}[1]{\ast^{#1}} \nc{\mal}{\mathcal{O}}
\nc{\tet}{\tilde{\ast}} \nc{\teot}{\tilde{\odot}}
\nc{\oex}{\overline{x}} \nc{\oey}{\overline{y}}
\nc{\oez}{\overline{z}} \nc{\oef}{\overline{f}}
\nc{\oea}{\overline{a}} \nc{\oeb}{\overline{b}}
\nc{\weast}[1]{\widetilde{\ast}^{#1}}
\nc{\weodot}[1]{\widetilde{\odot}^{#1}} \nc{\hstar}[1]{\star^{#1}}
\nc{\lae}{\langle} \nc{\rae}{\rangle}
\nc{\lf}{\lfloor}
\nc{\rf}{\rfloor}

\nc{\QQ}{{\mathbb Q}}
\nc{\RR}{{\mathbb R}} \nc{\ZZ}{{\mathbb Z}}

\nc{\cala}{{\mathcal A}} \nc{\calb}{{\mathcal B}}
\nc{\calc}{{\mathcal C}}
\nc{\cald}{{\mathcal D}} \nc{\cale}{{\mathcal E}}
\nc{\calf}{{\mathcal F}} \nc{\calg}{{\mathcal G}}
\nc{\calh}{{\mathcal H}} \nc{\cali}{{\mathcal I}}
\nc{\call}{{\mathcal L}} \nc{\calm}{{\mathcal M}}
\nc{\caln}{{\mathcal N}} \nc{\calo}{{\mathcal O}}
\nc{\calp}{{\mathcal P}} \nc{\calr}{{\mathcal R}}
\nc{\cals}{{\mathcal S}} \nc{\calt}{{\mathcal T}}
\nc{\calu}{{\mathcal U}} \nc{\calw}{{\mathcal W}} \nc{\calk}{{\mathcal K}}
\nc{\calx}{{\mathcal X}} \nc{\CA}{\mathcal{A}}

\nc{\fraka}{{\mathfrak a}} \nc{\frakA}{{\mathfrak A}}
\nc{\frakb}{{\mathfrak b}} \nc{\frakB}{{\mathfrak B}}
\nc{\frakc}{{\mathfrak c}}
\nc{\frakD}{{\mathfrak D}} \nc{\frakF}{\mathfrak{F}}
\nc{\frakf}{{\mathfrak f}} \nc{\frakg}{{\mathfrak g}}
\nc{\frakH}{{\mathfrak H}} \nc{\frakL}{{\mathfrak L}}
\nc{\frakM}{{\mathfrak S}} \nc{\bfrakM}{\overline{\frakM}}
\nc{\frakm}{{\mathfrak m}} \nc{\frakP}{{\mathfrak P}}
\nc{\frakN}{{\mathfrak N}} \nc{\frakp}{{\mathfrak p}}
\nc{\frakS}{{\mathfrak S}} \nc{\frakT}{\mathfrak{T}}
\nc{\frakX}{{\mathfrak X}}
\nc{\BS}{\mathbb{S
}}

\font\cyr=wncyr10 \font\cyrs=wncyr7
\nc{\li}[1]{\textcolor{red}{#1}}
\nc{\lir}[1]{\textcolor{red}{Li:#1}}
\nc{\yi}[1]{\textcolor{blue}{Yi: #1}}
\nc{\xing}[1]{\textcolor{purple}{Xing:#1}}
\nc{\revise}[1]{\textcolor{red}{#1}}

\nc{\ID}{{\rm I}}\nc{\lbar}[1]{\overline{#1}}\nc{\bre}{{\rm bre}}
\nc{\sd}{\cals}\nc{\rb}{\rm RB}\nc{\A}{\rm A}\nc{\LL}{\rm L}\nc{\tx}{\tilde{X}}
\nc{\col}{\Delta_{RT}}\nc{\mul}{m_{RT}}\nc{\ul}{u_{RT}}\nc{\epl}{\varepsilon_{RT}}
\nc{\hl}{H_{RT}}\nc{\arro}[1]{#1}\nc{\px}{P_{\tx}}\nc{\pw}{P_{\mathfrak{w}}}\nc{\pl}{B_\omega^+}
\nc{\pp}{\pl}\nc{\ppp}[1]{B^+(#1)}\nc{\dw}{\diamond_{\mathfrak{w}}}\nc{\dl}{\diamond_{\rm \ell}}
\nc{\ncshaw}{\sha^{{\rm NC}}_{\Omega}}\nc{\ncshal}{\sha^{{\rm NC}}_{{\rm RT}}}
\nc{\ver}{\rm V}\nc{\ld}{l}\nc{\del}{\Delta_{{\rm \ell}}}\nc{\epsl}{\epsilon_{{\rm \ell}}}
\nc{\uul}{u_{{\rm \ell}}}\nc{\oneh}{\mathbf{1}}\nc{\onew}{\mathbf{1}}
\nc{\etree}{1} \nc{\conc}{m_{RT}}
\nc{\hrtb}{H_{RT}(X\sqcup\Omega)} \nc{\hrts}{H_{RT}(X, \Omega)}\nc{\rts}{\mathcal{T}(X, \Omega)}\nc{\rfs}{\mathcal{F}(X, \Omega)} \nc{\ncshall}{\sha^{{\rm NC}}_{{\rm RT}}} \nc{\ldl}{\leq_{\mathrm{dl}}} \nc{\pla}{B_{\alpha}^{+}} \nc{\plb}{B_{\beta}^{+}}
\nc{\mapmonoid}{\frakM} \nc{\db}{\leq_{{\rm db}}} \nc{\mapm}[1]{\frakM(#1)} \nc{\mstar}{\frakM^\star(X)} \nc{\stars}[2]{#1|_{#2}}
\nc{\Id}{\mathrm{Id}} \nc\blw[1]{\lc#1\rc}\nc{\suba}[1]{|_{#1}} \nc{\Irr}{\mathrm{Irr}}
\nc\leqo{\leq_{{\rm db}}}\nc\odb{<_{{\rm db}}} \nc{\mpu}{u^{\ast}} \nc{\mpv}{v^{\ast}}
\nc{\brep}{\text{bre}_{L}} \nc{\good}{good\xspace} \nc{\sgood}{super good\xspace}  \nc{\Good}{Good\xspace}
\nc\ordc{<_{{\rm Dl}}} \nc\ordqc{\leq_{{\rm Dl}}}

\begin{document}

\title[New operated polynomial identities and Gr\"{o}bner-Shirshov bases]{
New operated polynomial identities and Gr\"{o}bner-Shirshov bases}
%

%
\author{Jinwei Wang}
\address{School of Mathematics and Physics, Lanzhou Jiaotong University, Lanzhou, Gansu 730070, P.\,R. China}
\email{wangjinwei@mail.lzjtu.cn}

\author{Zhicheng Zhu}
\address{School of Mathematics and Statistics, Lanzhou University, Lanzhou, Gansu 730000, P.\,R. China}
\email{zhuzhch16@lzu.edu.cn}

\author{Xing Gao$^*$}\thanks{*Corresponding author}
\address{School of Mathematics and Statistics, Key Laboratory of Applied Mathematics and Complex Systems, Lanzhou University, Lanzhou, Gansu 730000, P.\,R. China}
\email{gaoxing@lzu.edu.cn}

\date{\today}
\begin{abstract}
Quite recently, Bremner et al. introduced a new approach to Rota's Classification Problem
and classified some (new) operated polynomial identities. In this paper, we prove that all
operated polynomial identities classified by Bremner et al.
are Gr\"{o}bner-Shirshov.
\end{abstract}

\makeatletter
\@namedef{subjclassname@2020}{\textup{2020} Mathematics Subject Classification}
\makeatother
\subjclass[2020]{
16W99, 
16S10, 
08B20, 
13P10, 
05A05,   
}

\keywords{Rota's Classification Problem;  operated associative algebras; rewriting systems; Gr\"obner-Shirshov basis}

\maketitle

\tableofcontents

\setcounter{section}{0}

\allowdisplaybreaks

\section{Introduction}
\subsection{Rota's Classification Problem}
In the study of mathematics and mathematical physics, various linear operators---characterized by various
operator identities---played crucial roles. Inspired by this, Rota~\cite{Ro} posed the problem of
\begin{quote}
finding all possible \underline{algebraic identities} that can be satisfied by a linear operator on an \underline{algebra},
\label{qu:rota}
\end{quote}
henceforth called {\bf Rota's Classification Problem}. Here an algebra means an associative algebra.

Such operator identities interested to Rota included
\begin{eqnarray}
 \text{Endomorphism operator} &\quad&
d(x_1x_2)=d(x_1)d(x_2), \nonumber\\
 \text{Differential operator} &\quad&
d(x_1x_2)=d(x_1)x_2+x_1d(x_2), \nonumber\\
 \text{Average operator} &\quad&
P(x_1)P(x_2)=P(x_1P(x_2)),  \nonumber\\
\text{Inverse average operator} &\quad& P(x_1)P(x_2)=P(P(x_1)x_2),
\nonumber\\
 \text{(Rota-)Baxter operator of weight\ $\lambda$} &\quad&
 P(x_1)P(x_2)=P(x_1P(x_2)+P(x_1)x_2+\lambda x_1x_2),\nonumber\\
 &\quad &\quad \text{ where } \lambda \text{ is a fixed constant},\nonumber\\
\text{Reynolds operator}
  &\quad& P(x_1)P(x_2)=P(x_1P(x_2)+P(x_1)x_2-P(x_1)P(x_2)). \nonumber
\end{eqnarray}

After Rota posed Rota's Classification Problem, more and more linear operators have appeared, such as \begin{eqnarray}
\text{Differential operator of weight} ~ \lambda &\quad&
 d(x_1x_2)= d(x_1)x_2+x_1 d(x_2)+\lambda d(x_1)d(x_2), \nonumber\\
&\quad& \quad\text{ where } \lambda \text{ is a fixed constant},\nonumber\\
\text{Nijenhuis operator} &\quad& P(x_1)P(x_2)=P(x_1P(x_2)+P(x_1)x_2-P(x_1x_2)),\nonumber\\
\text{Leroux's TD operator} &\quad&
P(x_1)P(x_2)=P(x_1P(x_2)+P(x_1)x_2-x_1P(1)x_2),\nonumber\\
\text{Modified Rota-Baxter operator of weight} &\quad&
P(x_1)P(x_2)=P(x_1P(x_2)+P(x_1)x_2) + \lambda x_1x_2, \nonumber\\
\text{Modified differential operator of weight} &\quad&
d(x_1x_2)= d(x_1)x_2+x_1 d(x_2) + \lambda x_1x_2. \nonumber
\end{eqnarray}

In particular,  the endomorphism operator plays pivotal roles in Galois theory.
The differential operator is an algebraic abstraction of derivation in analysis and leads to differential algebra as an algebraic study of differential equations, which has been largely successful in many important areas \mcite{Ko, PS, Ri}.
The Rota-Baxter operator, originated from \mcite{Bax} in 1960 based on the probability study,
is to understand Spitzer's identity in fluctuation theory. The broad connections of Rota-Baxter operators with many areas in mathematics and mathematical physics are remarkable, such as the classical Yang-Baxter equation, operads, combinatorics, Hopf algebra and renormalization of quantum field theory \mcite{Ag, Bai, BBGN, Gub, GK, ZGG}.
Other linear operators have also been studied extensively~\mcite{BBGN, CGM, CK1, Mill, Nij, PZGL, Re, Ro2}.

\subsection{History in solving Rota's Classification Problem.}
There are two stages in our understanding of Rota's Classification Problem. The first stage is the construction of the
algebraic framework to consider algebraic identities in Rota's Classification Problem.
An easier case of algebraic identities satisfied by algebras is the noncommutative polynomial,
which is free and leads to the extensive study of polynomial identity (PI) rings in 1960s.
Since algebraic identities in Rota's Classification Problem involve linear operators, they
become more complicated and are realized as free objects in the category of algebras together with linear operators,
which are originated from Kurosh~\cite{Ku}.

The second stage in our understanding of Rota's Classification Problem is what distinguishes the operated polynomial identities (OPIs) satisfied by these above listed linear operators from the arbitrary OPIs? In other words, Rota apparently asked to identify ``good" OPIs that are worth of further study. Roughly speaking, since algebras in Rota's Classification Problem are associative algebras,
such OPIs looked for by Rota are compatible with associativity. In the process of characterization of these
compatibility, two special classes of OPIs are studied---differential type OPIs and Rota-Baxter type OPIs.

The study of differential type OPIs was carried out in~\mcite{GSZ},
which includes the classical differential OPI. Therein theories of Gr\"{o}bner-Shirshov bases and rewriting systems were applied successfully. Another important class of OPIs, namely Rota-Baxter type, was systematically studied in \mcite{ZGGS}. As to be expected from comparing integral calculus with differential calculus in analysis, the Rota-Baxter type OPIs are more challenging than the differential counterpart. In \mcite{ZGGS}, Rota-Baxter type OPIs were also characterized by Gr\"{o}bner-Shirshov bases and rewriting systems. An outstanding achievement of applications of Gr\"{o}bner-Shirshov bases and rewriting systems in the characterization of differential type OPIs and Rota-Baxter OPIs sheds a light to apply them to study general OPIs, which was carried out in~\cite{GG}.

Quite recently, Bremner et al. introduced a new approach to Rota's Classification Problem,
based on the rank of matrixes from OPIs~\mcite{BE}. They obtained six OPIs with degree 2 and multiplicity 1, and eighteen OPIs
and two parametrized families with degree 2 and multiplicity 2.
These operators include the derivation, average operator, inverse average operator,
Rota-Baxter operator of weight zero, Nijenhuis operator and some new operators.

In the present paper, we prove that OPIs classified in~\cite{BE} are Gr\"{o}bner-Shirshov
in the framework of~\cite{GG}, via the method of Gr\"{o}bner-Shirshov bases. In other words, we show that
OPIs classified in~\cite{BE} are ``good" OPIs searched in Rota's Classification Problem. Our study is a partial
answer of Rota's Classification Problem.

\subsection{Outline of the paper.} In Section 2, we first recall the construction of free operated algebras,
which is used to the first stage of our understanding of Rota's Classification Problem.
We then review the necessary concepts of Gr\"{o}bner-Shirshov bases and the notation of Gr\"{o}bner-Shirshov OPIs.
Section 3 is devoted to prove that all OPIs in~\cite{BE} are Gr\"{o}bner-Shirshov. Three monomial orders $\leq_{{\rm dt}}$, $\leqo$ and $\ordqc$ are recalled. Based on these three monomial orders, we first prove that all OPIs classified in~\cite{BE}
of degree 2 and multiplicity 1 are Gr\"{o}bner-Shirshov (Theorem~\mref{thm:1}). Then we show that
all OPIs in~\cite{BE} of degree 2 and multiplicity 2 are Gr\"{o}bner-Shirshov (Theorem~\mref{thm:2}).

\smallskip

\noindent
{\bf Notation.} Throughout this paper, let $\bfk$ be a unitary commutative ring,
which will be the base ring of all algebras and linear maps.
Since the leading monomials of the OPIs considered in this paper are not fixed if we involve the unity $\bfone$, we consider the case not involving the unity $\bfone$ throughout the paper. See Remarks~\mref{re:one} and~\mref{re:two} for more details.
By an algebra we mean a non-unitary associative algebra.
For a set $X$, we use $\bfk X$ to denote the free module on $X$.
Denote by $S(X)$ and $M(X)$ the free semigroup and free monoid on $X$, respectively.

\section{Operated algebras and Gr\"{o}bner-Shirshov bases}
\subsection{Operated algebras}
A.\,G. Kurosh~\cite{Ku} introduced firstly the concept of algebras with linear operators. In~\cite{Gop}, it was called operated algebras and the construction of free operated algebras was obtained. See also~\cite{BCQ}.
\begin{defn}
\begin{enumerate}
\item An {\bf operated semigroup} (resp. {\bf operated algebra}) is a semigroup (resp. algebra) $U$ together with a map (resp. linear map) $P_U: U\to U$.
\item A morphism from an operated semigroup (resp. algebra) $(U, P_U)$ to an operated semigroup (resp.  algebra) $(V,P_V)$ is a semigroup (resp.  algebra) homomorphism $f :U\to V$ such that $f \circ P_U= P_V \circ f$.  \label{de:mapset}
\end{enumerate}
\end{defn}

For a set $Y$, we denote
$$\lc Y\rc := \{ \lc y\rc \mid y\in Y \},$$
which is a disjoint copy of $Y$. Let $X$ be a set. We recall the construction of the free operated semigroup on $X$, proceeding via the finite stages $\frakM_n(X)$ recursively defined as follows.
For the initial step, we define
$$\frakM_0(X) := S(X)\,\text{ and }\, \frakM_1(X) := S(X \sqcup \lc \frakM_0(X)\rc).$$
Notice that the inclusion $X\hookrightarrow X \sqcup \lc \frakM_0\rc $ induces a monomorphism
$$i_{0}:\mapmonoid_0(X) = S(X) \hookrightarrow \mapmonoid_1(X) = S(X \sqcup \lc \frakM_0\rc  )$$
of semigroups through which we can identify $\mapmonoid_0(X) $ with its image in $\mapmonoid_1(X)$.
Assume that
$\frakM_{ n-1}(X)$ has been defined and the embedding
$$i_{n-2,n-1}\colon  \frakM_{ n-2}(X) \hookrightarrow \frakM_{ n-1}(X)$$
has been obtained for~$n\geq 2$. Consider the case of $n$. Define
\begin{equation*}
 \label{eq:frakn}
 \frakM_{ n}(X) := S \big( X\sqcup\lc\frakM_{n-1}(X) \rc\big).
\end{equation*}
Since~$\frakM_{ n-1}(X) = S \big(X\sqcup \lc\frakM_{ n-2}(X) \rc\big)$ is the free semigroup on the set $X\sqcup \lc\frakM_{ n-2}(X) \rc$,
the injection
$$  X \sqcup \lc\frakM_{ n-2}(X) \rc \hookrightarrow
    X \sqcup\lc \frakM_{ n-1}(X) \rc $$
induces a semigroup embedding
\begin{equation*}
    \frakM_{ n-1}(X) = S \big( X\sqcup \lc\frakM_{n-2}(X) \rc \big)
 \hookrightarrow
    \frakM_{ n}(X) = S\big( X\sqcup\lc\frakM_{n-1}(X) \rc \big).
\end{equation*}
Finally we define the semigroup
$$ \frakM (X):=\dirlim \frakM_{ n}(X) = \bigcup_{n\geq 0} \frakM_{ n}(X).$$
Let $\bfk\frakM(X)$ be the free
module with the basis $\frakM(X)$. Using linearity, the concatenation product on $\frakM(X)$
can be extended to a multiplication on $\bfk\frakM(X)$, turning $\bfk\frakM(X)$ into an algebra.
Define an operator
$$\lc\,\rc: \frakM(X) \to \frakM(X),\quad u \mapsto \lc u\rc.$$
By linearly, the operator $\lc\,\rc: \frakM(X) \to \frakM(X)$ can be extended to a linear operator $\lc\,\rc: \bfk \frakM(X) \to \bfk \frakM(X)$, turning $(\bfk\frakM(X),\lc\,\rc) $ into an operated algebra.

\begin{defn}Let $X$ be a set.
\begin{enumerate}
\item Elements in $\frakM(X)$ are called {\bf  bracketed words} or
{\bf  bracketed monomials on $X$}. When $X$ is finite, we
may also just list its elements, as $\frakM(x_1,x_2, x_3)$ if $X = \{x_1, x_2, x_3\}$.

\item Any $u\in \frakM(X)$ can be written uniquely as a product:
\begin{equation*}
u = u_1\cdots u_n, \text{ for some } n\geq 1,\, u_i\in X\sqcup \lc\frakM(X)\rc,\, 1\leq i\leq n.
 \end{equation*}
We call $u_i$ {\bf prime} and define $n$ to be the {\bf breadth} of $u$, denoted by $|u|$.

\item Let $\phi(x_1,\ldots,x_k)\in \bfk\mapm{X}$ with $k\geq 1$ and $x_1, \ldots, x_k\in X$. We call
$\phi(x_1,\ldots,x_k) = 0$ (or simply $\phi(x_1,\ldots,x_k)$) an
{\bf  operated polynomial identity (OPI)}.
\end{enumerate}
\end{defn}

\begin{lemma}{\bf \cite{BCQ, Gop}}
Let $X$ be a set.
\begin{enumerate}
\item The $(\frakM(X),\lc\, \rc)$ together with the natural embedding $i:X \to \frakM(X)$ is the free operated semigroup on $X$; and
\mlabel{it:mapsetm}
\item The $(\bfk\mapm{X},\lc\,\rc)$ together with the natural embedding $i: X \to \bfk\mapm{X}$ is the free operated  algebra on $X$. \mlabel{it:mapalgsg}
\end{enumerate}
\mlabel{lem:freetm}
\end{lemma}

Let $\phi= \phi(x_1,\ldots,x_k)\in \bfk\mapm{X}$ be an OPI. For any operated algebra $(R,P)$ and any map $$\theta:X\to R, \, x_i \mapsto r_i, i=1, \ldots, k,$$
using the universal property of $\bfk\mapm{x_1, \ldots, x_k}$ as a free operated algebra on $\{x_1,\ldots,x_k\}$, there is a unique morphism $\free{\theta}:\bfk\mapm{x_1,\ldots,x_k}\to R$ of operated algebras that extends the map $\theta$. We use the notation
\begin{equation*}
\phi(r_1,\ldots,r_k):= \free{\theta}(\phi(x_1,\ldots,x_k))
\label{eq:phibar}
\end{equation*}
for the corresponding {\bf evaluation} or {\bf substitution} of $\phi(x_1,\ldots,x_k)$ at the point $(r_1,\ldots,r_k)$. Informally, this is the element of $R$ obtained from $\phi$ upon
replacing every $x_i$ by $r_i$, $1\leq i\leq k$ and the operator $\lc\ \rc$ by $P$.
\begin{defn}
With the above notations, we say that $\phi(x_1,\ldots,x_k)=0$ (or simply $\phi(x_1,\ldots,x_k)$) is an {\bf OPI satisfied by $(R,P)$} if
$$\phi(r_1,\ldots,r_k)=0 \quad \text{for all } r_1,\ldots,r_k\in R.$$
In this case, we call $(R,P)$ (or simply $R$) a {\bf  $\phi$-algebra} and $P$ a {\bf $\phi$-operator}.
More generally, for a subset $\Phi\subseteq \bfk\mapm{X}$, we call $R$ (resp. $P$) a {\bf $\Phi$-algebra} (resp. {\bf $\Phi$-operator}) if $R$ (resp. $P$) is a $\phi$-algebra (resp. $\phi$-operator) for each $\phi\in \Phi$.
\label{de:phialg}
\end{defn}

\begin{exam}
\begin{enumerate}
\item  If $\phi = x_1x_2-x_2x_1$, then a $\phi$-algebra is a commutative algebra.

\item If $$\phi=\lc x_1 x_2\rc -\lc x_1\rc x_2 -x_1\lc x_2\rc,$$
then a $\phi$-algebra is simply a differential algebra.

\item If $$\phi= \lc x_1\rc \lc x_2\rc-\lc x_1\lc x_2\rc\rc -\lc \lc x_1\rc x_2\rc -\lambda \lc x_1x_2\rc,$$
then a $\phi$-algebra is exactly a Rota-Baxter algebra of weight $\lambda$.
\end{enumerate}
\end{exam}

Let $(R, P)$ be an operated algebra. An operated ideal of $R$ is an ideal closed under the operator $P$.
For $S\subseteq R$, the {\bf operated ideal} $\Id(S)$ of~$R$ generated by $S$ is defined to be the smallest operated ideal of $R$ containing $S$.
For $\Phi\subseteq \bfk\mapm{X}$ and a set $Z$, let $S_\Phi(Z) \subseteq \bfk\mapm{Z}$ denote the substitution set
\begin{equation*}
S_\Phi(Z) := \{\,{\phi}(r_1,\ldots,r_k) \mid r_1,\ldots,r_k\in \mapm{Z}, \phi(x_1,\ldots,x_k)\in \Phi\,\}.
\end{equation*}
The following well-known result exhibits the existence of a free $\Phi$-algebra.

\begin{prop} {\rm (\cite{Coh})} Let $X$ be a set and $\Phi\subseteq \bfk\frakM(X)$ a system of OPIs. Then for a set $Z$, the quotient operated algebra $\bfk\mapm{Z}/\Id(S_\Phi(Z))$ is the free $\Phi$-algebra on $Z$.
\mlabel{pp:frpio}
\end{prop}

\subsection{Gr\"{o}bner-Shirshov bases}
We provide some notations of Gr\"obner-Shirshov bases~\cite{BCQ, GGZ, GSZ}.
Let us begin with the concept of $\star$-word, which will be used frequently.

\begin{defn}
Let $Z$ be a set, $\star$ a symbol not in $Z$ and $Z^\star = Z \sqcup \{\star\}$.
\begin{enumerate}
\item By a {\bf $\star$-bracketed word} on $Z$, we mean any bracketed word in $\mapm{Z^\star}$ with exactly one occurrence of $\star$, counting multiplicities. The set of all $\star$-bracketed words on $Z$ is denoted by $\frakM^{\star}(Z)$.
\item For $q\in \mstar$ and $u \in  \frakM({Z})$, we define $\stars{q}{u}:= q|_{\star \mapsto u}$ to be the bracketed word on $Z$ obtained by replacing the symbol $\star$ in $q$ by $u$.

\item For $q\in \mstar$ and $s=\sum_i c_i u_i\in \bfk\frakM{(Z)}$, where $c_i\in\bfk$ and $u_i\in \frakM{(Z)}$, we define
\begin{equation*}
 q|_{s}:=\sum_i c_i q|_{u_i}\,. \vspace{-5pt}
\end{equation*}
\end{enumerate}
\mlabel{def:starbw}
\end{defn}

The $\star$-bracketed words can be used to characterize operated ideals in $\bfk\mapm{Z}$.

\begin{lemma}{\rm (\cite{BCQ, GSZ})}
Let $Z$ be a set and $S \subseteq \bfk\mapm{Z}$. Then the operated ideal generated by $S$ is
\begin{equation*}
\hspace{10pt}\Id(S) = \left\{\, \sum_{i} c_i q_i|_{s_i} \mid c_i\in \bfk,
q_i\in \frakM^{\star}(Z), s_i\in S \right\}.
\end{equation*}
\mlabel{lem:opideal}
\end{lemma}

As preparation, we expose some notations.

\begin{defn}
Let $Z$ be a set, $s \in \bfk\mapm{Z}$ and $\leq$ a linear order on $\frakM(Z)$.
\begin{enumerate}
\item Let $s \notin \bfk$. The {\bf leading monomial} $\lbar{s}$ of $s$ is the largest monomial appearing in $s$.
The {\bf leading coefficient $c_s$ of $s$} is the coefficient of $\lbar{s}$ in $s$.
\mlabel{item:monic}

\item If $s \in \bfk$, we define the {\bf leading monomial of $s$} to be $1$ and the {\bf leading coefficient of $s$} to be $c_s= s$.\mlabel{item:scalar}

\item The element $s$ is called {\bf monic  with respect to $\leq$} if $s\notin \bfk $ and $c_s=1$. A subset $S \subseteq \bfk\mapm{Z}$ is called {\bf monic with respect to $\leq$} if every $s \in S$ is monic with respect to $\leq$.
\mlabel{it:res}
\end{enumerate}
\mlabel{def:irrS}
\end{defn}

A monomial order is needed in the theory of Gr\"{o}bner-Shirshov bases,
which is a well-order compatible with all operations.

\begin{defn}
Let $Z$ be a set. A {\bf monomial order on $\frakM(Z)$} is a
well-order $\leq$ on $\frakM(Z)$ such that
\begin{equation}
u < v \Longrightarrow q|_u < q|_v, \,\text{ for } \, u, v \in \frakM(Z)\,\text{ and }\, q \in \frakM^{\star}(Z).
\mlabel{eq:morder}
\end{equation}
\mlabel{de:morder}
\end{defn}

\begin{remark}
Since $\leq$ is a well-order, we have $u < \lc u \rc$ for all $u \in \frakM(Z)$ by Eq.~\meqref{eq:morder}.
\end{remark}

There are two kinds of compositions to consider in Gr\"{o}bner-Shirshov bases.

\begin{defn}
Let $Z$ be a set, $\leq$ a monomial order on $\mapm{Z}$ and $f, g \in\bfk\mapm{Z}$ monic.
\begin{enumerate}
  \item  If there are $u, v,w\in \mapm{Z}$ such that $w = \lbar{f}u = v\lbar{g}$ with max$\{|\lbar{f}|, |\lbar{g}|\} < |w| <
|\lbar{f}| + |\lbar{g}|$, we call
$$(f,g)_w:=(f,g)^{u,v}_w:= fu - vg$$
the {\bf intersection composition of $f$ and $g$ with respect to $w$}.
\mlabel{item:intcomp}
  \item  If there are $w\in\mapm{Z}$ and $q\in\frakM^{\star}(Z)$ such that $w = \lbar{f} = q\suba{\lbar{ g}},$ we call
$$(f,g)_w:=(f,g)^q_w := f - q\suba{g}$$
the {\bf including composition of $f$ and $g$ with respect to $w$}.
\mlabel{item:inccomp}
\end{enumerate}
\mlabel{defn:comp}
\end{defn}

Now we are ready for the concept of Gr\"{o}bner-Shirshov bases.

\begin{defn}
Let $Z$ be a set, $\leq$ a monomial order on $\mapm{Z}$ and  $S\subseteq\bfk\mapm{Z}$ monic.
\begin{enumerate}
\item  An element $f\in\bfk\mapm{Z}$ is called {\bf  trivial modulo $(S, w)$} with $w\in \frakM(Z)$ if
$$f =\sum_ic_i q_i\suba{s_i} \, \text{ with } \lbar{q_i\suba{s_i}} < w, \,\text{ where } c_i\in\bfk, q_i\in\frakM^{\star}(Z), s_i\in S,$$
denoted by $f\equiv 0 \mod(S, w)$. We write $f \equiv g \mod(S, w)$ if $f-g \equiv 0 \mod(S, w)$.

\item   We call $S$ a {\bf Gr\"{o}bner-Shirshov basis} in $\bfk\mapm{Z}$ with respect
to $\leq$ if, for all pairs $f, g \in S$, every intersection composition of the form $(f, g)^{u,v}_w$
is trivial modulo $(S, w)$, and every including composition of the form $(f, g)^q_w$ is trivial modulo $(S, w)$.
\end{enumerate}
\end{defn}

We end this section with a characterization of Rota's Classification Problem via the method of Gr\"{o}bner-Shirshov bases~\cite{GG} .

\begin{defn}
Let $X$ be a set and $\Phi\subseteq \bfk\frakM(X)$ a system of OPIs. Let $Z$ be a set and $\leq$
a monomial order on $\frakM(Z)$. We call $\Phi$ {\bf Gr\"{o}bner-Shirshov  on $Z$ with respect to} $\leq$ if $S_\Phi(Z)$ is a Gr\"{o}bner-Shirshov basis in $\bfk \frakM(Z)$ with respect to $\leq$.
\mlabel{def:pgord}
\end{defn}

\section{Gr\"{o}bner-Shirshov operated polynomial identities}
In this section, we prove that all OPIs classified in~\cite{BE} are Gr\"{o}bner-Shirshov.
{\em In the rest of the paper}, in order to be consistent with the notations in~\mcite{BE},
we use $L$ to denote the operator $\lc \,\rc$.

\subsection{OPIs of degree 2 and multiplicity 1}
In this subsection, we prove that all OPIs of degree 2 and multiplicity 1 classified in~\cite{BE} are Gr\"{o}bner-Shirshov.
Here the degree means the number of variables in each term and the multiplicity is the the number of operators $L$ in each term.
For example, the OPI $$L(x_1 x_2) - x_1 L(x_2) - L( x_1) x_2$$
has degree 2 and multiplicity 1. Let us now recall these OPIs.

\begin{lemma}~\cite[Theorem~5.1]{BE}
Let $X$ be a set. Consider the (nonzero) OPI in $\bfk\frakM(X)$:
$$aL(x_1 x_2) + bL(x_1) x_2 + cx_1 L(x_2) \,\text{ with } a,b,c\in \bfk\,\text{ and }\, x_1, x_2\in X. $$
The matrix of consequences $M(R)$ has rank 14 or 17. Rank 14 occurs if and only
if the values of the coefficients (up to nonzero scalar multiples) correspond to one
of the following six OPIs:
\begin{align*}
&L(x_1 x_2), \quad L(x_1 x_2) -x_1 L(x_2), \quad L(x_1 x_2) - L( x_1) x_2,\\
&L(x_1 x_2) - x_1 L(x_2) - L( x_1) x_2, \quad  L( x_1) x_2, \quad x_1 L(x_2).
\end{align*}
\mlabel{lem:one}
\end{lemma}

Next we recall the monomial order $\leq_{\rm dt}$ on $\mathfrak{S}(Z)$~\cite{GSZ}. Let $(Z, \leq)$ be a well-ordered set.
For $u\in \frakM(Z)$, denote by $\deg_Z(u)$ the
number of $x \in Z$ in $u$ with repetition.
For any
$$u=u_1\cdots u_m\,\text{ and }\, v=v_1\cdots v_n, $$
where $u_i$ and $v_j$ are prime. Define
$u <_{\rm dt} v $ inductively on $\dep(u)+\dep(v)\geq 0$.
For the initial step of $\dep(u)+\dep(v) = 0$, we have $u,v\in S(Z)$ and define $u <_{\rm dt} v$ if $u<_{\rm deg-lex} v$, that is,
$$(\deg_Z(u), u_1, \ldots, u_m) < (\deg_Z(v), v_1, \ldots, v_n) \, \text{ lexicographically}.$$
Then $\leq_{\rm dt}$ is a monomial order on $S(Z)$~\cite{BN}.
For the induction step, if $u = \lc u'\rc$ and $v = \lc v' \rc$, define
$u <_{\rm dt} v \,\text{ if }\,  {u'} <_{\rm dt} {v'}.$
If $u \in Z$ and $v= \lc v'\rc$, define $u <_{\rm dt} v$.
Otherwise, define
$$u <_{\rm dt} v \, \text{ if }\, (\deg_Z(u),  u_1, \ldots, u_m) < (\deg_Z(v),  v_1, \ldots, v_n) \, \text{ lexicographically}.$$
Then $\leq_{\rm dt}$ is a monomial order on $\frakM(Z)$~\cite{GSZ}.

Now we arrive at our first main result of this paper.

\begin{theorem}
Let $X$ and $Z$ be sets. The six OPIs in $\bfk\frakM(X)$ of degree 2 and multiplicity 1 classified in~\cite{BE} listed in Lemma~\mref{lem:one} are respectively Gr\"{o}bner-Shirshov on $Z$ with respect to the monomial order $\leq_{{\rm dt}}$.
\mlabel{thm:1}
\end{theorem}

\begin{proof}
It is trivial that three monomial OPIs are respectively Gr\"{o}bner-Shirshov on $Z$ with respect to the monomial order $\leq_{{\rm dt}}$.
Other three OPIs are differential type OPIs, which are Gr\"{o}bner-Shirshov on $Z$ with respect to the monomial order $\leq_{{\rm dt}}$~\cite{GGZ}.
\end{proof}

\subsection{OPIs of degree 2 and multiplicity 2}
In this subsection, we turn to prove that all
OPIs of degree 2 and multiplicity 2 classified in~\cite{BE} are Gr\"{o}bner-Shirshov.
Let us first review these OPIs.

\begin{lemma}~\cite[Theorem~6.12]{BE}
Let $X$ be a set. The matrix $M(R)^t$ has rank 16 if and only if the
parameters $a, b, c, d, e, f$ correspond to one of the following OPIs in $\bfk\frakM(X)$:
\begin{align*}
&L^2(x_1x_2) - L^2(x_1) x_2 - x_1L^2(x_2), \\
&L^2(x_1x_2) - L^2(x_1) x_2 , \\
&L^2(x_1x_2) - x_1L^2(x_2), \\
&L^2(x_1x_2),  \\
&L^2(x_1) x_2 ,\\
&x_1L^2(x_2),
\end{align*}
where $x_1, x_2\in X$. The matrix $M(R)^t$ has rank 19 if and only if the parameters $a, b, c, d, e, f$
correspond to one of the following OPIs in $\bfk\frakM(X)$:
\begin{align*}
&L^2(x_1x_2) +L(x_1L(x_2)) + x_1L^2(x_2)\quad \quad \text{(New identity A (right))},  \\
& L^2(x_1x_2) +L(L(x_1)x_2) + L^2(x_1) x_2 \quad \quad \text{(New identity A (left))},\\
&L^2(x_1x_2) + dL(x_1L(x_2))  - (d+1)x_1L^2(x_2) \,\text{ with } d\in \bfk\setminus\{0\} \quad \quad \text{(New identity B (right))},  \\
& L^2(x_1x_2) + b L(L(x_1)x_2) -(b+1) L^2(x_1) x_2 \,\text{ with } b\in \bfk\setminus\{0\} \quad \quad \text{(New identity B (left))}, \\
&L^2(x_1x_2) + L^2(x_1)x_2 + x_1L^2(x_2) + 2L(x_1)L(x_2) - 2 L(L(x_1)x_2) - 2L(x_1L(x_2))\quad \quad \text{(New identity C)},\\
&L(L(x_1)x_2) - L^2(x_1)x_2 \quad \quad \text{($P_1$)},  \\
&L(L(x_1)x_2)  \quad \quad \quad \quad\quad \quad\,\,\text{($P_2$)},  \\
& L(x_1L(x_2)) - x_1L^2(x_2) \quad \quad \text{($P_3$)},\\
& L(x_1L(x_2)) \quad \quad \quad \quad\quad \quad\,\,\text{($P_4$)},\\
& L(x_1)L(x_2) \quad \quad \quad \quad\quad \quad\,\,\text{($P_5$)}, \\
& L(x_1)L(x_2) - L(L(x_1)x_2) - L(x_1L(x_2))  \quad \quad \text{(Rota-Baxter)},\\
& L(x_1)L(x_2) - L(L(x_1)x_2) - L(x_1L(x_2)) + L^2(x_1 x_2)  \quad \quad \text{(Nijenhuis)},\\
& L(x_1)L(x_2) - L(L(x_1)x_2) \quad \quad \text{(inverse average)},\\
& L(x_1)L(x_2) - L(x_1L(x_2))  \quad \quad \text{(average)},
\end{align*}
where $x_1, x_2\in X$.
\mlabel{lem:two}
\end{lemma}
Notice that the average (resp. inverse average) OPI was named the right (resp. left) average OPI in~\mcite{BE}.

Next we recall two monomial orders $\leqo$~\cite{ZGGS} and $\ordqc$~\cite{QC},
which will be used frequently in the remainder of the paper.
Let $(Z, \leq)$ be a well-ordered set. Extend the well order $\leq$ on $Z$
to the degree lexicographical order $\leq_{{\rm deg-lex}} $ on $M(Z)$.
Further we extend $\leq$ on $\frakM(Z)$. Every $u \in \frakM(Z)$ may be uniquely written as a product in the form
\begin{equation*}
u=u_0 L(\mpu_1) u_1 L(\mpu_2)u_2 \cdots L(\mpu_r) u_r
\end{equation*}
where $u_0,\cdots,u_r \in M(Z)$,  $\mpu_1,\cdots, \mpu_r \in \frakM_{n-1}(Z)$.
Denote by $\deg_L(u)$ the number of occurrence of $L$,
and define the $L$-breadth $\brep(u)$ of $u$ to be $r$.  For example, if $u = x_0 L(x_1) x_2 L(x_3 x_4) x_5$
with $x_0, \cdots, x_5\in Z$, we have $\deg_L(u) = 2$ and $\brep(u) = 2$.
Let $u,v\in \frakM(Z)$ and write them uniquely in the form:
$$u=u_0 L(\mpu_1) u_1 L(\mpu_2)u_2 \cdots L(\mpu_r) u_r\,\text{ and }\, v=v_0 L(\mpv_1) v_1 L(\mpv_2) v_2 \cdots L(\mpv_s) v_s,$$
where $u_0,\cdots,u_r$, $v_0,\cdots,v_s\in M(Z)$
and $\mpu_1,\cdots, \mpu_r, \mpv_1,\cdots,\mpv_s\in \frakM(Z)$.
We define $u\leqo v $ by induction on $\dep(u) + \dep(v)\geq 0$. For the initial step of $\dep(u) + \dep(v) = 0$, we have $u,v\in S(Z)$ and use the degree lexicographical order. For the induction step of $\dep(u) + \dep(v) \geq 1$, we define
\begin{equation*}
u\leqo v \Leftrightarrow
\left\{\begin{array}{ll}
\deg_L(u) < \deg_L(v), \\
\text{ or } \deg_L(u) = \deg_L(v)\, \text{ and }\, \brep(u) < \brep(v),\\
\text{ or } \deg_L(u) = \deg_L(v), \, \brep(u) = \brep(v) (=r)\, \text{ and }\\
 \,(\mpu_1,\cdots,\mpu_r,u_0,\cdots,u_r) \leq (\mpv_1,\cdots, \mpv_r, v_0,\cdots,v_r) \text{ lexicographically.}
\end{array}
\right.
\end{equation*}
Here $\mpu_i\leqo \mpv_i$ and $u_i\leqo v_i$ are compared by the induction hypothesis.
Then $\leqo$ is a monomial order on $\frakM(Z)$~\cite{ZGGS}.

Let $(Z, \leq)$ be a well-ordered set, and let $u=u_1\cdots u_m$ and $v=v_1\cdots v_n$ be in $\mathfrak{S}(Z)$, where $u_i$ and $v_j$ are prime. We define
$u\ordc v $ by induction on $\dep(u)+\dep(v)\geq 0$.
If $\dep(u)+\dep(v) = 0$, we get $u,v\in S(Z)$ and define $u\ordc v$ by
$u <_{\rm deg-lex} v$, that is, $$u \ordc v \, \text{ if }\, ({\rm deg}_Z(u),  u_1, \ldots, u_m) < ({\rm deg}_Z(v), v_1, \ldots, v_n) \, \text{ lexicographically}.$$
Suppose $\dep(u)+\dep(v) \geq 1$. If $u = \lc u'\rc$ and $v = \lc v' \rc$, define
$u\ordc v \,\text{ if }\, u' \ordc v'.$
Otherwise, define
$$u\ordc v \, \text{ if }\, ({\rm deg}_Z(u), |u|,  u_1, \ldots, u_m) <({\rm deg}_Z(v), |v|,  v_1, \ldots, v_n) \, \text{ lexicographically}.$$
Then $\ordqc$ is a monomial order on $\mathfrak{S}(Z)$~\cite{QC}.

As a preparation, we obtain the following result.

\begin{prop}
Let $X$ and $Z$ be sets. Then the following OPIs in $\bfk\frakM(X)$ are respectively Gr\"{o}bner-Shirshov on $Z$
with respect to the monomial orders $\leqo$ or $\ordqc$:
\begin{align}
&L^2(x_1x_2) - L^2(x_1) x_2 - x_1L^2(x_2),  \mlabel{eq:opi1}\\
&L^2(x_1x_2) - L^2(x_1) x_2 , \mlabel{eq:opi2}\\
&L^2(x_1x_2) - x_1L^2(x_2), \mlabel{eq:opi3}
\end{align}
where $x_1, x_2\in X$.
\mlabel{pp:1}
\end{prop}

\begin{proof}
We first consider Eq.~(\mref{eq:opi1}). Denote by
$$S = \{ L^2(x y) - L^2(x) y - xL^2(y)\mid x,y\in \frakM(Z)\}.$$
With respect to the monomial orders $\leqo$ or $\ordqc$, the leading monomial of $L^2(x y) - L^2(x) y - xL^2(y)$ is $L^2(x y)$.
Since the breadth of the leading monomial $L^2(x y)$ is $\bre(L^2(xy)) = 1$,
there are no intersection ambiguities. Further there are two cases of including ambiguities to consider.

\noindent{\bf Case 1.} We have
\begin{align*}
w =&\ \lbar{f} = L^2(p|_{L^2(xy)} z) = q|_{\lbar{g}}, \text{ where } x,y,z\in \frakM(Z), p,q\in \frakM^\star(Z), q= L^2(pz)\text{ and }\\
f =&\ L^2(p|_{L^2(xy)} z) - L^2(p|_{L^2(xy)}) z -  p|_{L^2(xy)} L^2(z), \\
g=&\ L^2(xy) - L^2(x) y - xL^2(y).
\end{align*}
The corresponding including composition is trivial mod $(S, w)$:
\begin{align*}
f - q|_g =&\ - L^2(p|_{L^2(xy)}) z -  p|_{L^2(xy)} L^2(z) + L^2(p|_{L^2(x)y} z) + L^2(p|_{xL^2(y)} z)\\
\equiv &\ - L^2(p|_{L^2(x)y}) z - L^2(p|_{xL^2(y)}) z -  p|_{L^2(x)y} L^2(z)  -  p|_{xL^2(y)} L^2(z)\\
&\ + L^2(p|_{L^2(x)y}) z + p|_{L^2(x)y} L^2(z) + L^2(p|_{xL^2(y)}) z + p|_{xL^2(y)} L^2(z)\\
\equiv&\ 0.
\end{align*}

\noindent{\bf Case 2.} We have
\begin{align*}
w =&\ \lbar{f} = L^2(x p|_{L^2(yz)}) = q|_{\lbar{g}}, \text{ where } x,y,z\in \frakM(Z), p,q\in \frakM^\star(Z), q= L^2(xp)\text{ and }\\
f =&\ L^2(x p|_{L^2(yz)}) - L^2(x) p|_{L^2(yz)} -  x L^2(p|_{L^2(yz)}), \\
g=&\ L^2(yz) - L^2(y) z - yL^2(z).
\end{align*}
Then the including composition of this ambiguity is trivial mod $(S, w)$:
\begin{align*}
f - q|_g =&\ - L^2(x) p|_{L^2(yz)} -  x L^2(p|_{L^2(yz)}) + L^2(x p|_{L^2(y)z}) +  L^2(x p|_{yL^2(z)})\\
\equiv &\   - L^2(x) p|_{L^2(y)z}  - L^2(x) p|_{yL^2(z)} -  x L^2(p|_{L^2(y)z}) -  x L^2(p|_{yL^2(z)})\\
 &\ + L^2(x) p|_{L^2(y)z} + x L^2(p|_{L^2(y)z}) +  L^2(x) p|_{yL^2(z)}+  x L^2(p|_{yL^2(z)})\\
\equiv&\ 0.
\end{align*}
So the OPI in Eq.~(\mref{eq:opi1}) is Gr\"{o}bner-Shirshov.

Next we consider Eq.~(\mref{eq:opi2}). Let
$$S = \{ L^2(x y) - L^2(x) y \mid x,y\in \frakM(Z)\}.$$
With a similar argument to the case of Eq.~(\mref{eq:opi1}), there are two cases to consider.

\noindent{\bf Case 3.} We have
\begin{align*}
w =&\ \lbar{f} = L^2(p|_{L^2(xy)} z) = q|_{\lbar{g}}, \text{ where } x,y,z\in \frakM(Z), p,q\in \frakM^\star(Z), q= L^2(pz)\text{ and }\\
f =&\ L^2(p|_{L^2(xy)} z) - L^2(p|_{L^2(xy)}) z , \\
g=&\ L^2(xy) - L^2(x) y .
\end{align*}
The corresponding including composition is trivial mod $(S, w)$:
\begin{align*}
f - q|_g =&\ - L^2(p|_{L^2(xy)}) z  + L^2(p|_{L^2(x)y} z) \\
\equiv &\ - L^2(p|_{L^2(x)y}) z  + L^2(p|_{L^2(x)y}) z \\
\equiv&\ 0.
\end{align*}

\noindent{\bf Case 4.} We have
\begin{align*}
w =&\ \lbar{f} = L^2(x p|_{L^2(yz)}) = q|_{\lbar{g}}, \text{ where } x,y,z\in \frakM(Z), p,q\in \frakM^\star(Z), q= L^2(xp)\text{ and }\\
f =&\ L^2(x p|_{L^2(yz)}) - L^2(x) p|_{L^2(yz)} , \\
g=&\ L^2(yz) - L^2(y) z.
\end{align*}
In this case, the corresponding including composition is trivial mod $(S, w)$ as follows:
\begin{align*}
f - q|_g =&\ - L^2(x) p|_{L^2(yz)}  + L^2(x p|_{L^2(y)z}) \\
\equiv &\   - L^2(x) p|_{L^2(y)z} + L^2(x) p|_{L^2(y)z} \\
\equiv&\ 0.
\end{align*}
So the OPI in Eq.~(\mref{eq:opi2}) is Gr\"{o}bner-Shirshov.
The proof of the OPI in Eq.~(\mref{eq:opi3}) is similar to the one in Eq.~(\mref{eq:opi2}).
\end{proof}

\begin{remark}
If we consider the case involving the unity $\bfone$, then the leading monomial of the OPI
$$L^2(x_1x_2) - L^2(x_1) x_2 - x_1L^2(x_2)$$
in Eq.~(\mref{eq:opi1}) is not necessary $L^2(x_1x_2)$. For example, if $x_2 = \bfone$, then
Eq.~(\mref{eq:opi1}) is $-x_1L^2(\bfone)$ whose leading monomial is not $L^2(x_1x_2) = L^2(x_1)$.
\mlabel{re:one}
\end{remark}

Next we turn to consider the new identity A.

\begin{prop}
Let $X$ and $Z$ be sets.
The following OPIs in $\bfk\frakM(X)$ are respectively Gr\"{o}bner-Shirshov on $Z$ with respect to the monomial orders $\leqo$ or $\ordqc$:
\begin{align}
&L^2(x_1x_2) +L(x_1L(x_2)) + x_1L^2(x_2) \quad\quad\text{(New identity A (right))},  \mlabel{eq:opi7}\\
& L^2(x_1x_2) +L(L(x_1)x_2) + L^2(x_1) x_2 \quad\quad\text{(New identity A (left))}, \mlabel{eq:opi8}
\end{align}
where $x_1, x_2\in X$.
\mlabel{pp:3}
\end{prop}

\begin{proof}
We only prove Eq.~({\mref{eq:opi7}}), as the case of Eq.~({\mref{eq:opi8}}) is a similar one.
Let $$S = \{ L^2(xy) +L(xL(y)) + xL^2(y) \mid x,y\in \frakM(Z)\}.$$
Notice that the leading monomial of $L^2(xy) +L(xL(y)) + xL^2(y)$ is $L^2(xy)$ with respect to the monomial orders $\leqo$ or $\ordqc$.
There are no intersection compositions, and there are two including compositions.

\noindent{\bf Case 1.} The ambiguity is
\begin{align*}
w =&\ \lbar{f} = L^2(p|_{L^2(xy)} z) = q|_{\lbar{g}}, \text{ where } x,y,z\in \frakM(Z), p,q\in \frakM^\star(Z), q= L^2(pz)\text{ and }\\
f =&\ L^2(p|_{L^2(xy)} z) + L(p|_{L^2(xy)} L(z)) + p|_{L^2(xy)} L^2(z), \\
g=&\ L^2(xy) +L(xL(y)) + xL^2(y),
\end{align*}
whose including composition is trivial mod $(S, w)$ as follows:
\begin{align*}
f - q|_g =&\ L(p|_{L^2(xy)} L(z)) + p|_{L^2(xy)} L^2(z) - L^2(p|_{L(xL(y))} z) - L^2(p|_{xL^2(y)} z)\\
\equiv&\ -L(p|_{L(xL(y))} L(z)) -L(p|_{xL^2(y)} L(z)) -  p|_{L(xL(y))} L^2(z) -  p|_{xL^2(y)} L^2(z)\\
&\ + L(p|_{L(xL(y))} L(z)) + p|_{L(xL(y))} L^2(z) + L(p|_{xL^2(y)} L(z)) + p|_{xL^2(y)} L^2(z)\\
\equiv&\ 0.
\end{align*}

\noindent{\bf Case 2.} The ambiguity is
\begin{align*}
w =&\ \lbar{f} = L^2(x p|_{L^2(yz)}) = q|_{\lbar{g}}, \text{ where } x,y,z\in \frakM(Z), p,q\in \frakM^\star(Z), q= L^2(xp)\text{ and }\\
f =&\ L^2(x p|_{L^2(yz)}) + L(x L(p|_{L^2(yz)})) + x L^2(p|_{L^2(yz)}), \\
g=&\ L^2(yz) +L(yL(z)) + yL^2(z).
\end{align*}
Then the including composition is trivial mod $(S, w)$ as follows:
\begin{align*}
f - q|_g =&\ L(x L(p|_{L^2(yz)})) + x L^2(p|_{L^2(yz)}) - L^2(x p|_{L(yL(z))}) - L^2(x p|_{yL^2(z)})\\
\equiv&\ -L(x L(p|_{L(yL(z))})) - L(x L(p|_{yL^2(z)})) - x L^2(p|_{L(yL(z))}) - x L^2(p|_{yL^2(z)})\\
&\ + L(x L(p|_{L(yL(z))})) + x L^2(p|_{L(yL(z))}) + L(x L(p|_{yL^2(z)})) + x L^2(p|_{yL^2(z)})\\
\equiv&\ 0,
\end{align*}
as needed.
\end{proof}

The following focuses on the new identity B.

\begin{prop}
Let $X$ and $Z$ be sets. The following OPIs in $\bfk\frakM(X)$ are respectively Gr\"{o}bner-Shirshov on $Z$
with respect to the monomial orders $\leqo$ or $\ordqc$:
\begin{align}
&L^2(x_1x_2) + dL(x_1L(x_2))  - (d+1)x_1L^2(x_2) \,\text{ with } d\in \bfk\setminus\{0\} \quad\quad\text{(New identity B (right))},  \mlabel{eq:opi9}\\
& L^2(x_1x_2) + b L(L(x_1)x_2) -(b+1) L^2(x_1) x_2 \,\text{ with } b\in \bfk\setminus\{0\} \quad\quad\text{(New identity B (left))}, \mlabel{eq:opi10}
\end{align}
where $x_1, x_2\in X$.
\mlabel{pp:4}
\end{prop}

\begin{proof}
By symmetry, it is enough to prove Eq.~({\mref{eq:opi9}}).
Let
$$S = \{L^2(x y) + dL(xL(y))  - (d+1)xL^2(y)\mid x,y\in \frakM(Z)\}.$$
The leading monomial of $L^2(x y) + dL(xL(y))  - (d+1)xL^2(y)$ is $L^2(x y)$ with respect to $\leqo$ or $\ordqc$.
There are no intersection compositions, and there are two including compositions.

\noindent{\bf Case 1.} The ambiguity is
\begin{align*}
w =&\ \lbar{f} = L^2(p|_{L^2(xy)} z) = q|_{\lbar{g}}, \text{ where } x,y,z\in \frakM(Z), p,q\in \frakM^\star(Z), q= L^2(pz)\text{ and }\\
f =&\ L^2(p|_{L^2(xy)} z) + dL(p|_{L^2(xy)} L(z)) -(d+1) p|_{L^2(xy)} L^2(z), \\
g=&\ L^2(xy) + d L(xL(y)) - (d+1) xL^2(y),
\end{align*}
whose including composition is trivial mod $(S, w)$ as follows:
\begin{align*}
f - q|_g =&\ d L(p|_{L^2(xy)} L(z)) - (d+1) p|_{L^2(xy)} L^2(z) - d L^2(p|_{L(xL(y))} z) + (d+1) L^2(p|_{xL^2(y)} z)\\
\equiv&\ -d^2L(p|_{L(xL(y))} L(z)) +d(d+1) L(p|_{xL^2(y)} L(z)) + d(d+1)  p|_{L(xL(y))} L^2(z) - (d+1)^2  p|_{xL^2(y)} L^2(z)\\
&\ + d^2 L(p|_{L(xL(y))} L(z)) - d(d+1) p|_{L(xL(y))} L^2(z) - d(d+1) L(p|_{xL^2(y)} L(z)) + (d+1)^2 p|_{xL^2(y)} L^2(z)\\
\equiv&\ 0.
\end{align*}

\noindent{\bf Case 2.} The ambiguity is
\begin{align*}
w =&\ \lbar{f} = L^2(x p|_{L^2(yz)}) = q|_{\lbar{g}}, \text{ where } x,y,z\in \frakM(Z), p,q\in \frakM^\star(Z), q= L^2(xp)\text{ and }\\
f =&\ L^2(x p|_{L^2(yz)}) + d L(x L(p|_{L^2(yz)})) -(d+1)x L^2(p|_{L^2(yz)}), \\
g=&\ L^2(yz) + d L(yL(z)) -(d+1) yL^2(z).
\end{align*}
The including composition is trivial mod $(S, w)$:
\begin{align*}
f - q|_g =&\ d L(x L(p|_{L^2(yz)})) -(d+1) x L^2(p|_{L^2(yz)}) - d L^2(x p|_{L(yL(z))}) +(d+1) L^2(x p|_{yL^2(z)})\\
\equiv&\ -d^2 L(x L(p|_{L(yL(z))})) +d(d+1) L(x L(p|_{yL^2(z)})) +d(d+1) x L^2(p|_{L(yL(z))}) - (d+1)^2 x L^2(p|_{yL^2(z)})\\
&\ + d^2 L(x L(p|_{L(yL(z))})) -d(d+1) x L^2(p|_{L(yL(z))}) -d(d+1) L(x L(p|_{yL^2(z)})) + (d+1)^2 x L^2(p|_{yL^2(z)})\\
\equiv&\ 0,
\end{align*}
as required.
\end{proof}

Now we are in a position to consider the new identity C.

\begin{prop}
Let $X$ and $Z$ be sets.
The following OPI in $\bfk\frakM(X)$ is Gr\"{o}bner-Shirshov on $Z$
with respect to the monomial order $\leq_{{\rm dt}}$:
\begin{align}
L^2(x_1x_2) + L^2(x_1)x_2 + x_1L^2(x_2) + 2L(x_1)L(x_2) - 2 L(L(x_1)x_2) - 2L(x_1L(x_2))\quad\quad\text{(New identity C)}.
\mlabel{eq:C}
\end{align}
where $x_1, x_2\in X$.
\mlabel{pp:5}
\end{prop}

\begin{proof}
Let $$S = \{L^2(x y) + L^2(x)y + xL^2(y) + 2L(x)L(y) - 2 L(L(x)y) - 2L(xL(y)) \mid x,y\in \frakM(Z)\}.$$
With respect to $\leq_{{\rm dt}}$, the leading monomial of $$L^2(x y) + L^2(x)y + xL^2(y) + 2L(x)L(y) - 2 L(L(x)y) - 2L(xL(y))$$
is $L^2(x y)$. Further there are no intersection compositions, and there are two including compositions.

\noindent{\bf Case 1.} The ambiguity is of the form
\begin{align*}
w =&\ \lbar{f} = L^2(p|_{L^2(xy)} z) = q|_{\lbar{g}}, \text{ where } x,y,z\in \frakM(Z), p,q\in \frakM^\star(Z), q= L^2(pz)\text{ and }\\
f =&\ L^2(p|_{L^2(xy)} z) + L^2(p|_{L^2(xy)}) z + p|_{L^2(xy)} L^2(z) +  2 L(p|_{L^2(xy)})L(z) -  2 L(L(p|_{L^2(xy)})z) - 2 L(p|_{L^2(xy)}L(z)), \\
g=&\ L^2(xy) + L^2(x)y + xL^2(y) + 2L(x)L(y) - 2 L(L(x)y) - 2L(xL(y)).
\end{align*}
Then the including intersection is trivial mod $(S, w)$:
\begin{align*}
f - q|_g =&\  L^2(p|_{L^2(xy)}) z + p|_{L^2(xy)} L^2(z) +  2 L(p|_{L^2(xy)})L(z) -  2 L(L(p|_{L^2(xy)})z) - 2 L(p|_{L^2(xy)}L(z))\\
&\ - L^2(p|_{L^2(x)y} z) - L^2(p|_{xL^2(y)} z) -  2 L^2(p|_{L(x)L(y)} z) + 2 L^2(p|_{L(L(x)y)} z) + 2 L^2(p|_{L(xL(y))} z)\\
\equiv&\ - L^2(p|_{L^2(x)y}) z - L^2(p|_{x L^2(y)}) z - 2 L^2(p|_{L(x)L(y)}) z + 2 L^2(p|_{L(L(x)y)}) z + 2 L^2(p|_{L(xL(y))}) z\\
&\ -p|_{L^2(x)y} L^2(z) - p|_{x L^2(y)} L^2(z) - 2 p|_{L(x)L(y)} L^2(z) + 2 p|_{L(L(x)y)} L^2(z) + 2 p|_{L(xL(y))} L^2(z)\\
&\ -2 L(p|_{L^2(x)y})L(z) -2 L(p|_{x L^2(y)})L(z) -4 L(p|_{L(x)L(y)})L(z) + 4 L(p|_{L(L(x)y)})L(z) + 4 L(p|_{L(xL(y))})L(z)\\
&\ +2 L(L(p|_{L^2(x)y})z) + 2 L(L(p|_{x L^2(y)})z) + 4 L(L(p|_{L(x)L(y)})z) - 4 L(L(p|_{L(L(x)y)})z) - 4 L(L(p|_{L(xL(y))})z)\\
&\ +2 L(p|_{L^2(x)y}L(z)) + 2 L(p|_{x L^2(y)}L(z)) + 4 L(p|_{L(x)L(y)}L(z)) - 4 L(p|_{L(L(x)y)}L(z)) - 4 L(p|_{L(xL(y))}L(z)) \\
&\ + L^2(p|_{L^2(x)y}) z +p|_{L^2(x)y} L^2(z) + 2 L(p|_{L^2(x)y})L( z) - 2 L( L(p|_{L^2(x)y})z) - 2 L(p|_{L^2(x)y}L( z))\\
&\ + L^2(p|_{xL^2(y)}) z +p|_{xL^2(y)} L^2(z) + 2 L(p|_{xL^2(y)})L( z) - 2 L(L(p|_{xL^2(y)}) z)- 2 L(p|_{xL^2(y)}L( z))\\
&\ +2L^2(p|_{L(x)L(y)}) z + 2p|_{L(x)L(y)} L^2(z) + 4 L(p|_{L(x)L(y)})L( z) - 4 L(L(p|_{L(x)L(y)}) z) - 4 L(p|_{L(x)L(y)}L( z))\\
&\ -2 L^2(p|_{L(L(x)y)}) z - 2p|_{L(L(x)y)} L^2(z) - 4 L(p|_{L(L(x)y)})L( z) + 4 L( L(p|_{L(L(x)y)})z) + 4 L(p|_{L(L(x)y)}L( z))\\
&\ -2L^2(p|_{L(xL(y))}) z - 2p|_{L(xL(y))} L^2(z) - 4 L(p|_{L(xL(y))})L( z) + 4 L(L(p|_{L(xL(y))}) z) + 4L(p|_{L(xL(y))}L( z))\\
&\equiv\ 0.
\end{align*}

\noindent{\bf Case 2.} The ambiguity is of the form:
\begin{align*}
w =&\ \lbar{f} = L^2(x p|_{L^2(yz)}) = q|_{\lbar{g}}, \text{ where } x,y,z\in \frakM(Z), p,q\in \frakM^\star(Z), q= L^2(xp)\text{ and }\\
f =&\ L^2(x p|_{L^2(yz)}) + L^2(x) p|_{L^2(yz)} +x L^2(p|_{L^2(yz)}) + 2 L(x)L( p|_{L^2(yz)}) - 2 L(L(x) p|_{L^2(yz)}) -2 L(xL( p|_{L^2(yz)})), \\
g=&\ L^2(yz) + L^2(y)z + yL^2(z) + 2L(y)L(z) - 2 L(L(y)z) - 2L(yL(z)).
\end{align*}
The corresponding including composition is trivial mod $(S, w)$:
\begin{align*}
f - q|_g =&\ L^2(x) p|_{L^2(yz)} + x L^2(p|_{L^2(yz)}) + 2 L(x )L( p|_{L^2(yz)}) - 2 L(L(x ) p|_{L^2(yz)}) - 2 L(xL( p|_{L^2(yz)}))\\
&\ - L^2(x p|_{L^2(y)z}) - L^2(x p|_{y L^2(z)}) - 2 L^2(x p|_{L(y)L(z)}) + 2 L^2(x p|_{L(L(y)z)}) + 2 L^2(x p|_{L(yL(z))})\\
\equiv&\ -L^2(x) p|_{L^2(y)z} -  L^2(x) p|_{y L^2(z)} - 2 L^2(x) p|_{L(y)L(z)} + 2 L^2(x) p|_{L(L(y)z)} + 2L^2(x) p|_{L(yL(z))}\\
&\ - x L^2(p|_{L^2(y)z})- x L^2(p|_{y L^2(z)}) - 2 x L^2(p|_{L(y)L(z)}) + 2 x L^2(p|_{L(L(y)z)}) + 2 x L^2(p|_{L(yL(z))})\\
&\ - 2 L(x )L( p|_{L^2(y)z}) - 2L(x )L( p|_{y L^2(z)}) - 4 L(x )L( p|_{L(y)L(z)}) + 4 L(x )L( p|_{L(L(y)z)}) + 4 L(x )L( p|_{L(yL(z))})\\
&\ + 2L(L(x ) p|_{L^2(y) z}) + 2L(L(x ) p|_{y L^2(z)}) + 4 L(L(x ) p|_{L(y)L(z)}) - 4 L(L(x ) p|_{L(L(y)z)}) - 4 L(L(x ) p|_{L(yL(z))})\\
&\ + 2 L(xL( p|_{L^2(y)z})) + 2 L(xL( p|_{y L^2(z)})) + 4 L(xL( p|_{L(y)L(z)})) - 4 L(xL( p|_{L(L(y)z)})) - 4 L(xL( p|_{L(yL(z))}))\\
&\ +L^2(x) p|_{L^2(y)z} + x L^2(p|_{L^2(y)z}) + 2 L(x)L( p|_{L^2(y)z}) - 2 L(L(x) p|_{L^2(y)z}) -  2L(xL( p|_{L^2(y)z}))\\
&\ + L^2(x) p|_{y L^2(z)} +x L^2(p|_{y L^2(z)}) + 2L(x)L( p|_{y L^2(z)}) - 2 L(L(x) p|_{y L^2(z)}) - 2 L(xL( p|_{y L^2(z)}))\\
&\ + 2 L^2(x) p|_{L(y)L(z)} + 2x L^2(p|_{L(y)L(z)}) + 4L(x)L( p|_{L(y)L(z)}) - 4 L(L(x) p|_{L(y)L(z)}) - 4 L(xL( p|_{L(y)L(z)}))\\
&\ - 2 L^2(x) p|_{L(L(y)z)} - 2x L^2(p|_{L(L(y)z)}) - 4 L(x)L( p|_{L(L(y)z)}) + 4 L(L(x) p|_{L(L(y)z)}) + 4 L(xL( p|_{L(L(y)z)}))\\
&\ -2 L^2(x) p|_{L(yL(z))} - 2x L^2(p|_{L(yL(z))}) - 4 L(x)L( p|_{L(yL(z))}) + 4 L(L(x) p|_{L(yL(z))}) + 4L(xL( p|_{L(yL(z))}))\\
\equiv&\ 0.
\end{align*}
This completes the proof.
\end{proof}

\begin{remark}
\begin{enumerate}
\item If we involve the unity $\bfone$, then the leading monomial of the OPI
$$L^2(x_1x_2) + L^2(x_1)x_2 + x_1L^2(x_2) + 2L(x_1)L(x_2) - 2 L(L(x_1)x_2) - 2L(x_1L(x_2))$$
in Eq.~(\mref{eq:C}) is not necessary $L^2(x_1x_2)$ with respect to the order $\leq_{{\rm dt}}$. For example, taking
$x_1 = \bfone$, then the above OPI in Eq.~(\mref{eq:C}) is $$L^2(\bfone)x_2 + 2L(\bfone)L(x_2) - 2L(L(\bfone)x_2),$$
whose leading monomial is not $L^2(x_1x_2)=L^2(x_2)$.

\item In Proposition~\mref{pp:5}, if we apply the monomial orders $\leqo$ or $\ordqc$, then the leading monomial of
$$ L^2(x y) + L^2(x)y + xL^2(y) + 2L(x)L(y) - 2 L(L(x)y) - 2L(xL(y))$$ is $L(x)L(y)$.
It induces a rewriting rule
$$L(x)L(y) \rightarrow -\frac{1}{2} L^2(x y) -\frac{1}{2} L^2(x)y -\frac{1}{2} xL^2(y) + L(L(x)y) + L(xL(y)).$$
Taking $y$ to be $L(y)$, we obtain an infinite rewriting process:
\begin{align*}
L(x)L^2(y) \rightarrow&\ -\frac{1}{2} L^2(x L(y)) -\frac{1}{2} L^2(x)L(y) -\frac{1}{2} xL^3(y) + L(L(x)L(y)) + L(x L^2(y))\\
 \rightarrow&\ -\frac{1}{2} L^2(x L(y)) + \frac{1}{4} L^2(L(x)y) + \frac{1}{4} L^3(x)y + \frac{1}{4} L(x)L^2(y)
 -\frac{1}{2}L(L^2(x)y) \\
 &\ - \frac{1}{2}L(L(x)L(y)) -\frac{1}{2} xL^3(y) + L(L(x)L(y)) + L(x L^2(y)),\\
\rightarrow&\ \cdots.
\end{align*}
Notice that the term $L(x)L^2(y)$ appears again in the right hand side.
\end{enumerate}
\mlabel{re:two}
\end{remark}

The following result is needed.

\begin{prop}
Let $X$ and $Z$ be sets.
The following OPIs in $\bfk\frakM(X)$ are respectively Gr\"{o}bner-Shirshov on $Z$ with respect to the monomial orders $\leqo$ or $\ordqc$:
\begin{align}
&L(L(x_1)x_2) - L^2(x_1)x_2 \quad \quad \text{($P_1$)} ,  \mlabel{eq:opi12}\\
& L(x_1L(x_2)) - x_1L^2(x_2) \quad \quad \text{($P_3$)} , \mlabel{eq:opi13}
\end{align}
where $x_1, x_2\in X$.
\mlabel{pp:6}
\end{prop}

\begin{proof}
By symmetry, it suffices to prove the case of Eq.~(\mref{eq:opi12}).
Let $$S = \{ L(L(x)y) - L^2(x)y \mid x,y\in \frakM(Z)\}.$$
The leading monomial of $L(L(x)y) - L^2(x)y$ is $L(L(x)y)$ with respect to $\leqo$ or $\ordqc$.
There are no intersection compositions, and there are three including compositions.

\noindent{\bf Case 1.} The ambiguity is of the form
\begin{align*}
w =&\ \lbar{f} = L(L(p|_{L(L(x)y)}) z) = q|_{\lbar{g}}, \text{ where } x,y,z\in \frakM(Z), p,q\in \frakM^\star(Z), q= L(L(p)z)\text{ and }\\
f =&\  L(L(p|_{L(L(x)y)}) z) - L^2(p|_{L(L(x)y)}) z, \\
g=&\ L(L(x)y) - L^2(x)y.
\end{align*}
Then the corresponding including composition is trivial mod $(S, w)$:
\begin{align*}
f-q|_g =&\ -L^2(p|_{L(L(x)y)}) z + L(L(p|_{L^2(x)y}) z) \\
\equiv&\ -L^2(p|_{L^2(x)y}) z + L^2(p|_{L^2(x)y}) z\\
\equiv&\ 0.
\end{align*}

\noindent{\bf Case 2.} The ambiguity is of the form
\begin{align*}
w =&\ \lbar{f} = L(L(x) p|_{L(L(y)z)}) = q|_{\lbar{g}}, \text{ where } x,y,z\in \frakM(Z), p,q\in \frakM^\star(Z), q= L(L(x)p)\text{ and }\\
f =&\  L(L(x) p|_{L(L(y)z)}) -  L^2(x) p|_{L(L(y)z)} , \\
g=&\ L(L(y)z) - L^2(y)z.
\end{align*}
The corresponding including composition is trivial mod $(S, w)$:
\begin{align*}
f-q|_g =&\ -  L^2(x) p|_{L(L(y)z)} + L(L(x) p|_{L^2(y)z}) \\
\equiv&\ -L^2(x) p|_{L^2(y)z} + L^2(x) p|_{L^2(y)z}  \\
\equiv&\ 0.
\end{align*}

\noindent{\bf Case 3.} The ambiguity is of the form
\begin{align*}
w =&\ \lbar{f} = L(L(L(x)y)z) = q|_{\lbar{g}}, \text{ where } x,y,z\in \frakM(Z), q= L(\star z)\in \frakM^\star(Z) \text{ and }\\
f =&\  L(L(L(x)y)z) - L^2(L(x)y)z , \\
g=&\ L(L(x)y) - L^2(x)y,
\end{align*}
whose corresponding including composition is trivial mod $(S, w)$:
\begin{align*}
f-q|_g =&\ -  L^2(L(x)y)z + L(L^2(x)yz) \\
\equiv&\ -L(L^2(x)y)z +L^3(x)yz\\
\equiv&\ -L^3(x)yz +L^3(x)yz\\
\equiv&\ 0.
\end{align*}
This completes the proof.
\end{proof}

In summary, we conclude the second main result of this paper.

\begin{theorem}
Let $X$ and $Z$ be sets.
All OPIs in $\bfk\frakM(X)$ of degree 2 and multiplicity 2 classified in~\cite{BE} listed in Lemma~\mref{lem:two} are respectively Gr\"{o}bner-Shirshov on $Z$ with respect to the monomial orders $\leqo$, $\ordqc$ or $\leq_{{\rm dt}}$.
\mlabel{thm:2}
\end{theorem}

\begin{proof}
The Rota-Baxter OPI, Nijenhuis OPI, average OPI and inverse average OPI are Rota-Baxter type OPIs, which are respectively Gr\"{o}bner-Shirshov on $Z$ with respect to the monomial order $\leqo$~\cite{ZGGS}. Further the monomial OPIs are respectively Gr\"{o}bner-Shirshov on $Z$ with respect to the monomial orders $\leqo$, $\ordqc$ and $\leq_{{\rm dt}}$. Finally, the remainder follows from Propositions~\mref{pp:1}, \mref{pp:3}, \mref{pp:4}, \mref{pp:5} and~\mref{pp:6}.
\end{proof}

\smallskip
\noindent
{\bf Acknowledgments.}
This work is supported by the National Natural Science Foundation of
China (Grant No. 12071191), the Natural Science Foundation of Gansu
Province (Grant No. 20JR5RA249) and the Natural Science Foundation of Shandong Province
(ZR2020MA002).

\end{document}